\documentclass{sn-jnl}
\usepackage[utf8]{inputenc}

\usepackage{amssymb}
\usepackage{manyfoot}
\usepackage{xcolor}
\usepackage{url}
\usepackage{amsmath}
\usepackage{amsthm}
\usepackage{proof}
\usepackage{ebproof}
\usepackage[scr=rsfs]{mathalpha}
\usepackage[all]{xy}
\usepackage{tikz}
\usetikzlibrary{calc,shapes.misc,shapes.geometric}

\usepackage{enumitem}
\usepackage{xcolor}

\newcounter{desccount}

\newcommand{\descref}[1]{\hyperref[#1]{#1}}

\newcommand{\At}{\text{At}}

\newcommand{\myhypertarget}[2]{%
  \phantomsection
  \hypertarget{#1}{\color{magenta}{#2}}%
  \expandafter\gdef\csname targettext@#1\endcsname{#2}%
}
\newcommand{\myhyperlink}[1]{%
  \hyperlink{#1}{\csname targettext@#1\endcsname}%
}

\newcommand{\derives}[1]{\vdash_{\!#1}}

\newcommand{\base}[1]{\mathscr{#1}}

\makeatletter
\def\slashedarrowfill@#1#2#3#4#5{%
  $\m@th\thickmuskip0mu\medmuskip\thickmuskip\thinmuskip\thickmuskip
   \relax#5#1\mkern-7mu%
   \cleaders\hbox{$#5\mkern-2mu#2\mkern-2mu$}\hfill
   \mathclap{#3}\mathclap{#2}%
   \cleaders\hbox{$#5\mkern-2mu#2\mkern-2mu$}\hfill
   \mkern-7mu#4$%
}
\def\rightslashedarrowfill@{%
  \slashedarrowfill@\relbar\relbar\mapstochar\rightarrow}
\newcommand\xslashedrightarrow[2][]{%
  \ext@arrow 0055{\rightslashedarrowfill@}{#1}{#2}}
\makeatother

\DeclareMathSymbol{\mhyphen}{\mathord}{AMSa}{"39}

\DeclareSymbolFont{bbsymbol}{U}{bbold}{m}{n}
\DeclareMathSymbol{\fatsemi}{\mathbin}{bbsymbol}{"3B}
\DeclareMathSymbol{\fatcomma}{\mathbin}{bbsymbol}{"2C}

\newcommand{\worlds}{\mathcal W}

\newcommand{\conssym}{\triangleright}
\newcommand{\cons}{\ifmmode\mathrel{\conssym}\else\mbox{$\conssym$}\fi}

\DeclareFontFamily{U}{min}{}
\DeclareFontShape{U}{min}{m}{n}{<-> udmj30}{}

\renewcommand{\derives}[2]{\vdash_{\base{#1}}#2}
\newcommand{\support}[2]{\Vdash_{\base{#1}}#2}
\newcommand{\kripke}[2]{\models_{\base{#1}}#2}

\newcommand\Oval[1]{
\tikz[baseline=(a.base)]{\node(a){\phantom{#1}};\node[draw,black,inner sep=2pt, rounded rectangle]{#1}}
}

\newcommand{\contn}[2]{\begin{bmatrix}#1\\#2\end{bmatrix}}

\newcommand{\db}[1]{[\![#1]\!]}

\newtheorem{theorem}{Theorem}                         
\newtheorem{definition}[theorem]{Definition}                                                     
\newtheorem{lemma}[theorem]{Lemma}                             
\newtheorem{example}[theorem]{Example}                 
\newtheorem{proposition}[theorem]{Proposition}

\newcommand{\fillBox}{\hfill$\Box$}

\begin{document}

\title{Continuations and Completeness in  Proof-theoretic Semantics}

\author*[1,3]{\fnm{Tao} \sur{Gu}}\email{tao.gu22@gmail.com}
\equalcont{These authors contributed equally to this work.}

\author*[2,3]{\fnm{David} \sur{Pym}}\email{david.pym@sas.ac.uk}
\equalcont{These authors contributed equally to this work.}

\author*[4]{\fnm{Eike} \sur{Ritter}}\email{e.ritter@bham.ac.uk}
\equalcont{These authors contributed equally to this work.}

\author*[5]{\fnm{Edmund} \sur{Robinson}}\email{e.p.robinson@qmul.ac.uk}
\equalcont{These authors contributed equally to this work.}

\affil*[1]{\orgname{University of Southampton}, \orgaddress{Highfield, Southampton, SO17 1BJ, UK}}

\affil[2]{\orgdiv{Institute of Philosophy}, \orgname{School of Advanced Study, University of London}, \orgaddress{\street{Malet Street}, 
\city{London}, \postcode{WC1E 7HU}, 
\country{UK}}}

\affil*[3]{
\orgname{University College London}, \orgaddress{\street{Gower Street}, 
\city{London}, \postcode{WC1E 6BT}, 
\country{UK}}}

\affil[4]{\orgdiv{Department of Computer Science}, \orgname{University of Birmingham}, \orgaddress{\street{Edgbaston}, 
\city{Birmingham}, \postcode{B15 2TT}, 
\country{UK}}}

\affil[5]{\orgdiv{School of Electronic Engineering and Computer Science}, \orgname{Queen Mary, University of London}, \orgaddress{\street{Mile End Road}, 
\city{London}, \postcode{E1 4NS}, 
\country{UK}}}

\abstract{


This is a short paper about the relationship between logic and computation. More specifically, it is about a relationship between the completeness proof for intuitionistic propositional logic within the form of proof-theoretic semantics that is known as base-extension semantics and a fundamental idea from the theory of computation called continuation-passing semantics. The latter is  explained herein both in terms of reduction in natural deduction and the lambda calculus and in terms of proof-search. The relationship between completeness and continuations is explored through an analysis of Sandqvist's proof of the completeness theorem as seen from the mathematical perspective of Kripke's and Heyting's semantics. Our analysis can be seen to reveal how syntactic representations of continuations embody intensional semantical intuitions about the relationship between their meaning and use. These intuitions are made precise using the tools of proof-theoretic semantics.
}

\maketitle

\section{Introduction} \label{sec:intro} 

The core of this paper is an analysis of Sandqvist's \cite{Sandqvist2015base} proof of completeness of a form of base-extension semantics for intuitionistic propositional logic, contrasting with Piecha and Schroeder-Heister's demonstration of the incompleteness of a range of alternatives \cite{Piecha2019incompleteness}. 
Base-extension semantics is a form of proof-theoretic semantics in which the meanings of logical connectives are explained relative to atomic `bases'. Each base establishes a primitive derivation system on a set of atomic propositions. The base-extension semantics then determines the validity of propositional formulae built from these atomic propositions via logical connectives. 

Our purpose in this paper is to analyse Sandqvist's argument as seen through the lens of the established semantics given by Kripke and Heyting, using these to relate Sandqvist's definition to a form of continuation-passing semantics. Our analysis is presented in an essentially denotational way. It is, borrowing phrases from Schroeder-Heister's introduction to proof-theoretic semantics \cite{SEP-PtS}, presented in a basically extensional fashion, making it essentially model-theoretic, and therefore not  in this case addressing the intensional and epistemological aspects that are core to the proof-theoretic analysis. Our purpose here is a mathematical analysis of a completeness proof that enables us to link the structures in use to techniques taken from other areas, notably areas that deal with related issues to proof-theoretic semantics. 

In Section~\ref{sec:PtSbasics}, we explain, following Sandqvist \cite{Sandqvist2015base}, the elements of base-extension semantics in the setting of intuitionistic propositional logic. In Section~\ref{sec:KripkeHeyting}, we discuss, in the context of \cite{Sandqvist2015base}, the relationship between Kripke models and Heyting algebras. In Section~\ref{sec:Completeness1}, we explain Sandqvist's completeness proof in a way that will support our ongoing discussion. Section~\ref{sec:Continuations} introduces the ideas, from the semantics of computation, of a \emph{continuation} and of a \emph{CPS-transform}. 
In Section~\ref{sec:ContinuationsAndLogic}, we explain that the ideas of continuation and continuation-passing arise in proof theory --- specifically, through the manipulation and computation of proofs. Section~\ref{sec:Nuclei} considers 
the structure of Heyting algebras in more detail, showing how Sandqvist's semantics can be seen to arise through a specific interpretation. Then with all these components available, we are able to explain, in Section~\ref{sec:Completeness2}, how Sandqvist's semantics implicitly employs a CPS-transform --- this is the main result of the paper. We conclude, in Section~\ref{sec:proof-relevance}, with a discussion of the consequences of our result for a proof-relevant semantics in this setting.

\section{Atomic bases and proof-theoretic Semantics}\label{sec:PtSbasics}


Base-extension semantics is founded on the notion of a \emph{base}, the function of which is to give a consequence relation on a collection of atomic propositions. We will assume for the purpose of this paper that we are given a fixed countably infinite set $\At$ of atomic propositions. Bases, $\base{A}$, $\base{B}$, \dots \;  are presented as sets of deductive rules, similar in structure to natural deduction proof rules, but only involving specific atoms.




Sandqvist \cite{Sandqvist2015base} takes the general form of a rule to be:
\[
((P_1 \Rightarrow a_1),\dots,(P_n \Rightarrow a_n) 
    \Rightarrow b)
\]
where $a_1 , \dots , a_n$ and $b$ are atoms and each 
$P_i$ is a possibly empty finite set of atoms. 

This inductively defines a consequence relation $\derives{B}$, the \emph{derivability} in $\base{B}$ of an atom from a set of atoms: 
\[
    \begin{array}{rl}
    (\mbox{{Ref}}) & P , a \derives{B} a \\ 
    (\mbox{{App}}) & \mbox{For any $((P_1 \Rightarrow a_1),\dots,(P_n \Rightarrow a_n) \Rightarrow b)\in \base{B}$, and finite set of atoms $Q$,} \\ 
    & \mbox{if $Q , P_1 \derives{B} a_1$ \dots\, $Q , P_n \derives{B} a_n$, then $Q \derives{B} b$}
    \end{array}
\] 
Here, the comma between (sets of) atoms is the union as standard in proof theory.
The relations $\derives{B}{}$ are indeed consequence relations: 

\smallskip 

\begin{lemma}\label{lem:derives-is-consequence} If $\base{B}$ is a base, then
    \begin{enumerate}
        \item for any atom $a$, $a\derives{B}{a}$
        \item for any atoms $a$, $b$ and sets of atoms $P$, $Q$, 
        if $P\derives{B}{a}$ and $a,Q\derives{B}{b}$, then $P,Q\derives{B}{b}$, and 
        \item for any atom $a$ and sets of atoms $P$, $Q$, if $P\subseteq Q$ and 
        $P\derives{B}{a}$, then $Q\derives{B}{a}$. 
    \end{enumerate}
\end{lemma}
\begin{proof}
    The first statement is immediate from (Ref), and the others are proved by induction. For example, statement 2 is proved by induction over the length of derivation of $a,Q\derives{B}{b}$. If this is an instance of (Ref), then there are two cases: $a=b$ and $a\neq b$, both immediate. If however it ends with an application of (App) using the rule $((P_1 \Rightarrow a_1),\dots,(P_n \Rightarrow a_n) \Rightarrow b)$, then we have $a,Q,P_1\derives{B}{a_1}$, \dots 
    $a,Q,P_n\derives{B}{a_n}$, all by shorter derivations. Hence $P,Q,P_1\derives{B}{a_1}$,\dots, $P,Q,P_n\derives{B}{a_n}$, and the result follows. 
\end{proof}

Other authors have considered more restricted forms of atomic rule; see, for example, Piecha \cite{Piecha2016completeness} for a discussion. But the more general form used by Sandqvist is needed to allow atomic analogues of all the usual rules for natural deduction. 

\smallskip 

\begin{example}
\begin{enumerate}
    \item A rule $(\Rightarrow b)\in\base{B}$, with no hypotheses, says directly that $b$ is derivable in the base ($\derives{B} b$), and can be seen as encoding the assertion of $b$ in any context $Q$: 
    \[\begin{prooftree}
        \hypo{}
        \infer1{Q\derives{B} b}
    \end{prooftree}\]
    \item A rule $((\Rightarrow a_1),(\Rightarrow a_2)\Rightarrow b)\in\base{B}$, where all the hypotheses are unconditional, can be seen as encoding a proof rule: 
    \[\begin{prooftree}
        \hypo{Q \derives{B} a_1}
        \hypo{Q \derives{B} a_2}
        \infer2{Q \derives{B} b}
    \end{prooftree}\]
    \item A rule $((\Rightarrow a),(b_1\Rightarrow c),(b_2\Rightarrow c) \Rightarrow c)\in\base{B}$ can be seen as encoding an analogue of the $\vee$-elimination proof rule: 
    \[\begin{prooftree}
        \hypo{Q\derives{B} a}
        \hypo{Q,b_1\derives{B} c}
        \hypo{Q,b_2\derives{B} c}
        \infer3{Q\derives{B} c}
    \end{prooftree}\]
\end{enumerate}   
\item More generally, the rule $((P_1 \Rightarrow a_1),\dots,(P_n \Rightarrow a_n) \Rightarrow b)$ can be written in natural deduction style: 
   \[\begin{prooftree}
        \hypo{Q,P_1 \derives{B} a_1}
        \hypo{\dots}
        \hypo{Q,P_n\derives{B} a_n}
        \infer3{Q\derives{B} b}
    \end{prooftree}\]
and $Q\derives{B} b$ is derivable iff it is the result of a complete proof tree (i.e., a proof tree with no open assumptions). 
\end{example}
 
Since bases are given by sets of rules, the inclusion of these sets extends to a natural notion of the inclusion of one base in another:
$\base{A} \subseteq \base{B}$. This in turn gives the following  monotonicity property for derivability, proved by induction: 

\smallskip 

\begin{lemma}\label{lemm:derives-monotonic-in-base}
If $\base{A} \subseteq \base{B}$ and $P \derives{A}{a}$, then $P \derives{B}{a}$. \fillBox
\end{lemma}

\smallskip 

Moreover, derivability satisfies a form of deduction theorem: 

\smallskip 

\begin{lemma}\label{lem:derives-deduction}
    The following are equivalent, for any base $\base{B}$, atom $a$, and sets of atoms $P$ and $Q$:
    \begin{enumerate}
        \item $P,Q\derives{B}{a}$
        \item For all bases $\base{C}\supseteq\base{B}$, if $Q\derives{C}{p}$ holds for all $p\in P$, then $Q\derives{C}{a}$
        \item $Q\derives{B'}{a}$ where $\base{B'} = \base{B} \cup \{ (\Rightarrow p) \mid p\in P\}$. 
    \end{enumerate}
\end{lemma}
\begin{proof}
    (1$\Rightarrow$2) is immediate from monotonicity of $\derives{D}{}$ in the base $\base{D}$ (Lemma~\ref{lemm:derives-monotonic-in-base}), and the fact that it is a consequence relation (Lemma~\ref{lem:derives-is-consequence}). 3 is an instance of 2. 1 follows from 3 by induction over the derivation of 
    $Q\derives{B'}{a}$. 
\end{proof}

Derivability for atoms is now extended to a notion of support for propositions built from atoms using the standard logical connectives.
Sandqvist's definition for this is: 

\smallskip  

\begin{definition}[support]
\label{def:sandqvist-support}
    The support relation, in which base $\base{B}$ \emph{supports} the proposition $\phi$, written $\support{B}{\phi}$, is defined by recursion on the structure of $\phi$ as follows: 
    \[
    \begin{array}{lr@{\quad}c@{\quad}l}
    (\mbox{\rm At}) & \support{B}{a} & \mbox{iff} & \mbox{
    $\derives{B}{a}$, for atom $a$} \\ 
    (\wedge) & \support{B}{\phi\wedge\psi} & \mbox{iff} & 
    \mbox{$\support{B}{\phi}$ and $\support{B}{\psi}$} \\ 
    (\supset) & \support{B}{\phi\supset\psi} &  \mbox{iff} & 
    \phi \support{B}{\psi} \\ 
    (\vee) &  \support{B}{\phi\vee\psi} & \mbox{iff} & \mbox{for every atom $a$ and every $\base{C}\supseteq\base{B}$, if $\phi\support{C}{a}$ and $\psi\support{C}{a}$,} \\
    & & & \mbox{then $\support{C}{a}$} \\ 
    (\bot) & \support{B}{\bot} & \mbox{iff} & \mbox{$\support{B}{a}$, for every atom $a$}
    \end{array}
    \]
    Here the clauses $(\supset)$ and $(\vee)$ refer to a relation of inference, which is denoted using the same symbol $\support{\base{B}}{}$ with abuse of notation: 
    \[
    \begin{array}{lr@{\quad}c@{\quad}l}
    (\mbox{\rm Inf}) & \mbox{For $\Theta \neq \emptyset$, $\Theta\support{B}\phi$} & \mbox{iff} &  \mbox{for every $\base{C}\supseteq\base{B}$, if $\support{C}{\theta}$ for every $\theta\in\Theta$,} \\ 
    & & & \mbox{then $\support{C}\phi$}
    \end{array} 
    \] 
    \fillBox
\end{definition}

\smallskip 

\begin{definition}[validity] \label{def:validity} 
The consequence relation of \emph{validity}, $\Vdash$, is given by 
$\Gamma \Vdash \phi$ iff, for all bases $\mathcal{B}$, $\Gamma \Vdash_\mathcal{B} \phi$.  
\fillBox
\end{definition}

\smallskip 

Note that clause \emph{(At)} states that, for an atom $a$, $\support{B}{a}$ iff $\derives{B}{a}$, but this extends to consequence. 
\begin{lemma}\label{lem:support-eq-derives-atoms}
    For any set $P$ of atoms, and atom $a$, $P\support{B}{a}$ iff 
    $P\derives{B}{a}$. 
\end{lemma}
\begin{proof}
    This follows immediately from (At), (Inf), and Lemma~\ref{lem:derives-deduction}.
\end{proof}

Given the importance of intensionality in base-extension semantics, we have reproduced Sandqvist's definition exactly. But some points are worth noting. Sandqvist does not give an interpretation of 
$\top$, but could have used
\[
\begin{array}{l@{\quad}r@{\quad}c@{\quad}l}
    (\top) & \support{B}{\top} & \mbox{always} & 
\end{array}
\]
The definition of inference (Inf) is directly analogous to Kripke implication, where worlds are bases and the partial order between them is inclusion. As a result the interpretation of $(\supset)$ is monotonic in the base. Other rules are also monotonic, either enforcing this directly ($\vee$), inheriting this from monotonicity for atoms ($\At$,$\bot$), or via induction over their components ($\wedge$). 

\smallskip 

\begin{lemma}
    If $\support{B}{\phi}$ and $\base{B}\subseteq\base{C}$, then 
    $\support{C}{\phi}$. \fillBox
\end{lemma}

\smallskip 

As a result, we have $\support{B}{\phi}$ if and only if for every $\base{C}\supseteq \base{B}$, $\support{C}{\phi}$, which is the right hand side of (Inf) in the case that $\Theta$ is empty. Hence, if we allowed inference from an empty set of hypotheses, then we would get a concept extensionally equivalent to support, and the condition that $\Theta$ be non-empty in (Inf) can be seen as redundant. 

Sandqvist notes that the interpretations he gives for $\wedge$ and 
$\supset$ are `traditional'.  They are the ones used by most 
authors in the field (see {e.g.,} Piecha \cite{Piecha2016completeness} 
for a range of approaches). The interpretations of $\bot$ and $\vee$ are more novel. Sandqvist ascribes the interpretation for $\bot$ to Dummett \cite{Dummett1991} and others, and indicates that the interpretation for $\vee$ is loosely based on the elimination rule, developing previous work of Prawitz \cite{Prawitz2005} and Ferreira \cite{Ferreira2006} (cf. \cite{Prawitz1971ideas,Prawitz2006meaning}). 

Sandqvist \cite{Sandqvist2015base} establishes the soundness and completeness 
of the natural deduction calculus NJ (with 
consequence relation denoted $\vdash$) with respect to the base-extension semantics of intuitionistic propositional logic: 

\smallskip 

\begin{theorem}[soundness and completeness] \label{sec:s-and-c}
$\Gamma \vdash \phi$ \,iff\, $\Gamma \Vdash \phi$ \fillBox
\end{theorem}

\smallskip 

Sandqvist's proof of completeness for this interpretation uses the familiar method of encoding objects in more primitive forms. He encodes arbitrary propositions as chosen atoms and sets up a special base $\base{N}$ that allows him to ensure that the derivation of these atoms mimics deduction in an intuitionistic natural deduction system. So derivation corresponds to intuitionistic provability. Then he shows that the semantics he has chosen forces propositions to be logically equivalent to their corresponding atoms, whenever bases are required to extend $\base{N}$. From this completeness follows. 

Herein, we intend to examine Sandqvist's semantics and these constructions in the light of similar constructions for Kripke models and Heyting algebras, bringing out links to techniques from other areas, notably the use of continuations. 

\section{Kripke models and Heyting algebras}\label{sec:KripkeHeyting}

In Section~\ref{sec:PtSbasics}, we noted the monotonicity property and indicated its link to Kripke models. Now we make that connection more concrete. We suppose we are 
given a fixed set $\At$ of atoms and a fixed set $\worlds$ of bases over it, ordered by inclusion. This forms the frame for a Kripke model, in which the truth of an atom at a world is given by derivability and then the truth of a proposition at a world is defined by the standard Kripke clauses: 
\[ 
\begin{array}{lr@{\quad}c@{\quad}l}
    \mbox{$(\At)$} & \kripke{B}{a} & \mbox{iff} & 
    \mbox{$\derives{B}{a}$, atom $a$ } \\  
    \mbox{$(\top)$} & \kripke{B}{\top} & \mbox{always} &  \\  
    \mbox{$(\wedge)$} & \kripke{B}{\phi\wedge\psi} & \mbox{iff} &  
    \mbox{$\kripke{B}{\phi}$ and $\kripke{B}{\psi}$} \\
    \mbox{$(\supset)$} & \kripke{B}{\phi\supset\psi} & \mbox{iff} &  
    \mbox{for every $\base{C}\supseteq\base{B}$, if $\kripke{C}{\phi}$, 
    then $\kripke{C}{\psi}$}  \\ 
    \mbox{$(\vee)$} & \kripke{B}{\phi\vee\psi} & \mbox{iff} & \mbox{$\kripke{B}
    {\phi}$ or $\kripke{B}{\psi}$} \\ 
    \mbox{$(\bot)$} & \kripke{B}{\bot} & \mbox{never} &  \\ 
\end{array}
\]
There is also a form of internal consequence at a base:
\[ 
\begin{array}{lr@{\quad}c@{\quad}l}
\mbox{(Inf)} & \Theta\kripke{B}\phi & \mbox{iff} & 
\mbox{for every $\base{C}\supseteq\base{B}$, if $\kripke{C}{\theta}$ for every $\theta\in\Theta$, then $\kripke{C}\phi$}
\end{array}
\]
As a result, the clause for $\supset$ could be reformulated: 
\[ 
\begin{array}{lr@{\quad}c@{\quad}l}
\mbox{($\supset$)} & \kripke{B}{\phi\supset\psi} &  \mbox{iff} & \phi\kripke{B}{\psi} 
\end{array}
\]
This semantics coincides with Sandqvist's, except for the clauses for $\bot$ and $\vee$, as he notes. Moreover it is also the semantics used by a number of other authors, see Piecha \cite{Piecha2016completeness}. Note that again we have ({cf.} Lemma~\ref{lem:support-eq-derives-atoms}): 

\smallskip 

\begin{lemma}\label{lem:kripke-eq-derives-atoms}
    For any set $P$ of atoms, and atom $a$, $P\kripke{B}{a}$ iff $P\derives{B}{a}$. \fillBox  
\end{lemma}

\smallskip 

We can further re-frame this definition, giving the semantics of a proposition as the set of bases that support it: 
\[
\db{\theta} = \{\base{B} \in \worlds \mid\  \kripke{B}{\theta}\}
\]
We view this as a mathematical convenience, not as a philosophical statement intended to capture the intrinsic meaning of the proposition. It does some violence to the underlying approach of proof-theoretic semantics, which is to give meaning to propositions by giving a simple account of circumstances that justify them, not by flattening that to a list of situations. Nevertheless, we contend that it is useful to see how PtS approaches appear through this lens. 

This semantic viewpoint gives the possible interpretations of propositions as sets of bases. Monotonicity implies that every set $\db{\phi}$ is an upwards-closed set of worlds. The upwards-closed sets of worlds form a lattice, $\worlds^\uparrow$, in fact a complete Heyting algebra, and so we are led to the interpretation of connectives via the standard interpretation of intuitionistic propositional logic in such a structure (e.g., \cite{Dummett1977,MM1992,Vickers-TvL1989}). 
We have: 
\[
\begin{array}{lrcl}
    \mbox{(At)} & \db{a} & = & \mbox{$\{\base{B} \in \worlds \mid \base{B}\vdash a \}$} \\ 
    \mbox{($\top$)} & \db{\top} & = & \worlds \\ 
    \mbox{($\bot$)} & \db{\bot} & = &  \emptyset \\ 
    \mbox{($\wedge$)} & \db{\phi\wedge\psi} & = & \db{\phi}\cap\db{\psi} \\ 
    \mbox{($\supset$)} & \db{\phi\supset\psi} & = & \mbox{$\bigcup \,\{\, U \mid U$ is upwards closed and $U \cap \db{\phi} \subseteq \db{\psi} \,\}$}  \\   
    \mbox{($\vee$)} & \db{\phi\vee\psi} & = &  \db{\phi}\cup\db{\psi} 
\end{array}
\] 

Sandqvist's interpretation differs from this, and that naturally raises the questions: 
\begin{itemize}
    \item[--] Why does it differ? 
    \item[--] Is there some coherent approach in which it does fit?
\end{itemize}
In order to motivate our answer to these, we look closer into Sandqvist's completeness proof.






\section{Proving completeness}
\label{sec:Completeness1}

Sandqvist's proof of completeness uses an encoding of arbitrary 
propositions as atoms, and of the propositional deductive system into a base. Starting with a collection of atoms, \At, for every proposition 
$\phi$ defined over \At, we introduce a new atom \Oval{$\phi$} 
(Sandqvist uses the notation $\phi^\flat$). The 
aim is to force the atom \Oval{$\phi$} to be semantically equivalent 
to the proposition $\phi$. Thus \Oval{$a\vee b$} will be an atom 
semantically equivalent to the disjunction of the two atoms $a$ and 
$b$. Similarly for contexts $\Gamma = \{\phi_1,\dots,\phi_n\}$ we can define $\mbox{\Oval{$\Gamma$}} = \{
\mbox{\Oval{$\phi_1$}},\dots,\mbox{\Oval{$\phi_n$}}\}$.
Sandqvist is a bit more delicate than this. His completeness 
argument works at the level of a given single logical consequence $\Phi \vdash \psi$. In order 
to handle this, he only needs to consider subformulae of the consequence, 
thus allowing a restriction to a finite number of additional atoms and a finite base. Like Sandqvist, for an atom $a\in\At$ we take $\mbox{\Oval{$a$}}=a$. We call the resulting set of atoms $\At^\ast$. 

The special base $\base{N}$ is defined over $\At^\ast$ and consists of rules mimicking those of Natural Deduction. For instance, if 
\[\begin{prooftree}
    \hypo{\begin{array}{c}\\ 
    \varphi\vee\psi
    \end{array}} 
    \hypo{\begin{array}{c}[\varphi]\\
    \theta\end{array}}
    \hypo{\begin{array}{c}[\psi]\\
    \theta
    \end{array}}
    \infer3{\theta}
    \end{prooftree}
\]
is a valid instance of $\vee$-elimination, then $\base{N}$ contains 
\[
((\Rightarrow\mbox{\Oval{$\varphi\vee\psi$}}), 
  (\mbox{\Oval{$\varphi$}}\Rightarrow\mbox{\Oval{$\theta$}}), 
  (\mbox{\Oval{$\varphi$}}\Rightarrow\mbox{\Oval{$\theta$}}) 
  \Rightarrow\mbox{\Oval{$\theta$}})
  \]
the atomic rule corresponding to: 
\[\begin{prooftree}
    \hypo{\begin{array}{c}\\ 
    \Oval{$\varphi\vee\psi$}
    \end{array}
    }
    \hypo{\begin{array}{c}[\Oval{$\varphi$}]\\ 
       \Oval{$\theta$}\end{array}}
    \hypo{\begin{array}{c}[\Oval{$\psi$}]\\
        \Oval{$\theta$}\end{array}}
    \infer3{\Oval{$\theta$}}
    \end{prooftree}\]
The special base $\base{N}$ consists of all such translates of all the possible instances of all of the natural deduction rules. Sandqvist minimises the size of this base by restricting to atoms representing subformulae of a given sequent. As a result, his base is finite.  

Note that $\Oval{$a$} = a$ is a natural and handy choice, but it is not essential. If we had not identified atoms $a$ and \Oval{$a$}, then we would also have added rules enforcing that equivalence: $((\Rightarrow a)\Rightarrow\mbox{\Oval{$a$}})$ and 
$((\Rightarrow\mbox{\Oval{$a$}})\Rightarrow a)$.

Sandqvist's proof hinges on his translation and base $\base{N}$ having two key properties, $(\dagger)$, which expresses the semantic equivalence between a formula and its translate (see section \ref{sec:Completeness2}), and $(\ddagger)$, which expresses the equivalence between derivability in $\base{N}$ of atomic translates and provability of the original formulae in NJ, and in our notation is:
\[\mbox{$(\ddagger)$ for any $P$ and $a$, if $P \vdash_\base{N} a$ and if \Oval{$\Gamma$} $=$ $P$ and \Oval{$\phi$} $=$ $a$, then $\Gamma \vdash \phi$ }\]
\smallskip

\begin{proposition}[Sandqvist \cite{Sandqvist2015base}]
\label{prop:validity-reflection}
   For any propositions $\phi_1,\dots,\phi_n$, and $\psi$: $\mbox{\Oval{$\phi_1$}},\dots,\mbox{\Oval{$\phi_1$}}\derives{N}{\mbox{\Oval{$\psi$}}}$
    iff 
$\phi_1,\dots,\phi_n\vdash\psi$ is derivable in intuitionistic propositional logic. 
\end{proposition}
\begin{proof}
    The core idea is that a derivation $\mbox{\Oval{$\phi_1$}},\dots,\mbox{\Oval{$\phi_n$}}\derives{N}{\mbox{\Oval{$\psi$}}}$ expressed in tree form maps to an intuitionistic proof simply by removing all the ovals (mapping the function 
    $\mbox{\Oval{$\phi$}} \mapsto \phi$), and an intuitionistic proof to a derivation by doing the reverse. 
    
    Technically, the proof in each direction is by induction, in one direction on the length of the derivation in $\base{N}$, and in the other on the structure of the proof in NJ. 
\end{proof}

Note that for this to work when $\base{N}$ is restricted to subformulae of a sequent, then we need the subformula property for provability in intuitionistic logic. 

Gentzen's seminal work \cite{Gentzen1934,gentzen1969investigations} already contains the 
idea that the introduction rules express the core meaning of a 
connective and the elimination rules in some sense express the 
completeness of this ({cf.} Dummett \cite{Dummett1991}). There is a body of work establishing to what 
extent we can systematically derive elimination rules from certain 
forms of introduction rules, so-called `generalized elimination' 
rules ({cf.} \cite{dyckhoff2015some}). 

One way of expressing this completeness is that we can prove the identity axiom for the application of a connective from the identity axioms for the formulae to which it is applied, and hence that the identity axiom for all formulae follows from that for atoms only. We will see the importance of this later when we use a refinement of these arguments to establish the equivalence of a proposition
$\phi$ and its atomic representative \Oval{$\phi$}. This is easiest seen in calculi like Gentzen's LJ \cite{gentzen1969investigations} (we've removed some structural rules)  
\[
\begin{prooftree}
  \hypo{}
  \infer1{\varphi\to\varphi}
  \infer1{\varphi\to\varphi\vee\psi}
  \hypo{}
  \infer1{\psi\to\psi}
  \infer1{\psi\to\varphi\vee\psi}
  \infer2{\varphi\vee\psi\to \varphi\vee\psi}
\end{prooftree} 
\hspace{4em}\begin{prooftree}
  \hypo{}
  \infer1{\varphi\to\varphi}
  \hypo{}
  \infer1{\psi\to\psi}
  \infer2{\varphi,\psi\to\varphi\wedge\psi}
  \infer1{\varphi\wedge\psi\to \varphi\wedge\psi}
\end{prooftree}
\]
These correspond to natural deduction proofs 
\[
\begin{prooftree}
  \hypo{\varphi\vee\psi}
  \hypo{[\varphi]}
  \infer1{\varphi\vee\psi}
  \hypo{[\psi]}
  \infer1{\varphi\vee\psi}
  \infer3{\varphi\vee\psi}
\end{prooftree} 
\hspace{4em}\begin{prooftree}
  \hypo{\varphi\wedge\psi}
  \infer1{\varphi}
  \hypo{\varphi\wedge\psi}
  \infer1{\psi}
  \infer2{\varphi\wedge\psi}
\end{prooftree}
\]
which make the point less clear. 

This form of completeness implies that the logical rules determine the 
connectives up to logical equivalence. So, for example, if $\vee_1$ and $\vee_2$ both have the introduction and elimination rules for 
$\vee$, then, replacing parts of the proofs above with the analogues for $\vee_1$ and $\vee_2$, we have 
\[
\begin{prooftree}
  \hypo{}
  \infer1{\varphi\to\varphi}
  \infer1{\varphi\to\varphi\vee_2\psi}
  \hypo{}
  \infer1{\psi\to\psi}
  \infer1{\psi\to\varphi\vee_2\psi}
  \infer2{\varphi\vee_1\psi\to \varphi\vee_2\psi}
\end{prooftree} 
\hspace{4em}\mbox{or}\hspace{4em}\begin{prooftree}
  \hypo{\varphi\vee_1\psi}
  \hypo{[\varphi]}
  \infer1{\varphi\vee_2\psi}
  \hypo{[\psi]}
  \infer1{\varphi\vee_2\psi}
  \infer3{\varphi\vee_2\psi}
\end{prooftree} 
\]

This feeds into a failed attempt to prove completeness for base-extension semantics of Intuitionistic Propositional Logic using the Kripke interpretations given above. 

We want to prove the equivalence of a formula with the corresponding atom by induction on the structure of the formula. 

We have $\phi\wedge\psi \to \mbox{\Oval{$\phi\wedge\psi$}}$:
\[
\begin{prooftree}
    \hypo{\phi\wedge\psi}
    \infer1[$\wedge E$]{\phi}
    \infer1[Ind]{\Oval{$\phi$}}
    \hypo{\phi\wedge\psi}
    \infer1[$\wedge E$]{\psi}
    \infer1[Ind]{\Oval{$\psi$}}
    \infer2[$\base{N}$]{\Oval{$\phi\wedge\psi$}}
\end{prooftree}
\]
In particular, the last step, marked with $\base{N}$, uses the following rule in $\base{N}$:
\[
    ( (\Rightarrow \mbox{\Oval{$\phi$}}), (\Rightarrow \mbox{\Oval{$\psi$}}) \Rightarrow \mbox{\Oval{$\phi \land \psi$}} )
\]
We will omit the details of the application of such $\base{N}$-rules in the following cases.

This is to be read as a semantic argument in which the steps are 
validated through the soundness of intuitionistic rules ($\wedge E$), induction (Ind) and the soundness of the rules of $\base{N}$ for atoms. Note the structural similarity to the previous discussion of uniqueness of $\wedge$. 

Conversely, we have $\mbox{\Oval{$\phi\wedge\psi$}}\to \phi\wedge\psi$:
\[
\begin{prooftree}
    \hypo{\Oval{$\phi\wedge\psi$}}
    \infer1[$\base{N}$]{\Oval{$\phi$}}
    \infer1[Ind]{\phi}
    \hypo{\Oval{$\phi\wedge\psi$}}
    \infer1[$\base{N}$]{\Oval{$\psi$}}
    \infer1[Ind]{\psi}
    \infer2[$\wedge I$]{\phi\wedge\psi}
\end{prooftree}
\]

Similarly, we can validate $\phi\vee\psi \to \mbox{\Oval{$\phi\vee\psi$}}$:
\[
\begin{prooftree}
    \hypo{\phi\vee\psi}
    \hypo{\phi}
    \infer1[Ind]{\Oval{$\phi$}}
    \infer1[$\base{N}$]{\Oval{$\phi\vee\psi$}}
    \hypo{\psi}
    \infer1[Ind]{\Oval{$\psi$}}
    \infer1[$\base{N}$]{\Oval{$\phi\vee\psi$}}
    \infer3[$\vee E$]{\Oval{$\phi\vee\psi$}}
\end{prooftree}
\]

However, we run into problems with $\mbox{\Oval{$\phi\vee\psi$}}\to \phi\vee\psi$:
\[
\begin{prooftree}
    \hypo{\Oval{$\phi\vee\psi$}}
    \hypo{\Oval{$\phi$}}
    \infer1[Ind]{\phi}
    \infer1[$\vee I$]{\phi\vee\psi}
    \hypo{\Oval{$\psi$}}
    \infer1[Ind]{\psi}
    \infer1[$\vee I$]{\phi\vee\psi}
    \infer3[??]{\phi\vee\psi}
\end{prooftree}
\]
This is not a valid deduction. It fails at the last step 
because $\phi\vee\psi$ is not an atom, and hence the last deduction 
is not an instance of an atomic rule. In other words, the following rule that would pass the last step is not an atomic rule:
\[
    (\Rightarrow (\mbox{\Oval{$\phi \lor \psi$}}), (\mbox{\Oval{$\phi$}} \Rightarrow \phi \lor \psi), (\mbox{\Oval{$\psi$}} \Rightarrow \phi \lor \psi) \Rightarrow \phi \lor \psi )
\]
Nor do we have any other rule that would allow us to perform this last step. 
For our strategy to work, we 
have to find a way round this, and that is what Sandqvist did in his 
proof of completeness. In order to understand his approach we take a 
closer look at the structure of proof rules. 

The proof of equivalence for conjunction just given goes through because the rules only feature components of the conjunction: 
\[\begin{prooftree}
    \hypo{\phi}\hypo{\psi}\infer2{\phi \wedge \psi}
\end{prooftree}\hspace{3em}
\begin{prooftree}
    \hypo{\phi \wedge \psi}\infer1{\phi}
\end{prooftree}\hspace{3em}
\begin{prooftree}
    \hypo{\phi \wedge \psi}\infer1{\psi}
\end{prooftree}\]
The rules for $\supset$ and $\top$ have a similar property, and hence we can find proofs of equivalence for these connectives too. 
However,  the elimination rule for disjunction features a formula $\chi$ 
which is not syntactically related to the disjunction: 
\[\begin{prooftree}
    \hypo{\phi}\infer1{\phi \vee \psi}
\end{prooftree}\hspace{3em}
\begin{prooftree}
    \hypo{\psi}\infer1{\phi \vee \psi}
\end{prooftree}\hspace{3em}
\begin{prooftree}
    \hypo{\begin{array}{c} \\ 
    \phi \vee \psi 
    \end{array}
    }
    \hypo{\begin{array}{c}[\phi]\\
    \chi\end{array}}
    \hypo{\begin{array}{c}[\psi]\\
    \chi\end{array}}
    \infer3{\chi}
\end{prooftree}\]
It is this proposition $\chi$ that is causing us the problem. We need to look more closely at it. What $\chi$ represents is the idea that our final goal in constructing the proof is not local, but is some proposition $\chi$, potentially very different from the proposition we are currently working on. This idea also arises in the area of programming in the form of a `continuation', where it has been widely studied, notably for us in the context of giving continuation-based semantics to programs and continuation passing style transformations. 

\section{Continuations: an introduction}\label{sec:Continuations}




In his historical survey on the origins of the concept of continuation \cite{reynolds1993discoveries}, Reynolds notes that concepts recognizable as a form of continuation have been introduced a considerable number of times; for example as `a method of program transformation (into continuation-passing style), a style of definition interpreters 
\cite{Reynolds1972}, and a style of denotational semantics (in the sense of Scott and Strachey). Reynolds goes on to say: `In each of these settings, by representing ``the meaning of the rest of the program'' as a function or procedure, continuations provide an elegant description of a variety of language constructs, including call by value and goto statements.'

So, a \emph{continuation} is essentially the `rest of a computation' reified as an object that we can manipulate. 
To put this into context, if 
we want to understand the computational meaning of an expression in 
a programming language, then we face some important restrictions. 
First, the identifiers in the expression have to be resolved 
and given meaning. Essentially, we are working with closed terms, or 
at least open terms in an environment where the identifiers are 
bound to closed terms. Second, languages have restrictions on what 
kind of expression can actually be run as a program. Many programming 
languages specify that a runnable program is a function `\texttt{main}';  
for example, $C$, where \texttt{main} returns an integer, Java, where it returns void, or Rust, where it has to be called `\texttt{main}', but where there is no requirement on the type that it returns.  

As a result, in order to give a computational semantics for program fragments, we need a way to put them into a complete runnable setting. 
Moreover, most practical programming language have features that produce non-standard flow of control. Originally these included jumps and goto's, but they still include low level features such as break, and higher-level ones such as exceptions. Continuations can be used to explain the meaning of such features. 

In a continuations-based account, functions or procedures have an implicit additional argument, the continuation. This can be seen as a generalization and reification of the return address. If the function returns normally, then its return value is given as a parameter to the continuation supplied as an argument. If not, then the continuation is ignored, and another one is used instead. For example (simplifying a little) in programs with exceptions, the function of an exception handler is to provide an alternative continuation that can be used in the event of code generating that particular exception. So if code terminates normally, then the function returns using the standard continuation supplied as a parameter. And if it raises an exception it returns using the continuation provided by the exception handler. 

In a typed functional setting, the continuations for an expression of type $A$ can be represented as functions $A \to R$, where $R$ is the result type of a complete program. As a result, the semantics of an expression of type $A$ is a function from continuations to results of type $(A \to R) \to R$. 

The reification of continuations can be taken further, with translations from programs that do not manipulate continuations as data objects to programs that do. This is called a \emph{CPS-transform} --- that is, transform to continuation-passing style. Such a transformation can specify an evaluation mechanism for the source language. 

Taking lambda calculus as a basic programming language, we have the standard Plotkin-Fischer call-by-value transformation \cite{plotkin1975call}: 
\begin{eqnarray*}
  \db{x} & =  & \lambda k. k x\\
  \db{\lambda x. e }& = & \lambda k. k (\lambda x. \db{e})\\
  \db{f\, e} & = & \lambda k. \db{f} (\lambda f. \db{e} (\lambda e. f e k))
\end{eqnarray*}
Note that each of these translations is an expression with head $\lambda$ expecting a continuation argument. Moreover, if we provide an expression that also has head lambda as that continuation, and follow the standard functional programming restriction of not evaluating under lambda, then the term we are evaluating has only a single redex. This can be seen in an example reduction. Consider the term $f\, e$, where $f$ evaluates to $\lambda x. e'$, and $e$ evaluates to a value $v$, whose CPS transform is $\lambda k. kv$. We evaluate $f\, e$
with continuation $\lambda x. ?$. Then, putting the active continuation in blue, we have:
\begin{eqnarray*}
    \db{f\, e} {\color{blue}(\lambda x. ?)} & = & (\lambda k. \db{f} (\lambda f. \db{e} (\lambda e. f e k))) {\color{blue}(\lambda x. ?)}\\
    & \leadsto & \db{f} {\color{blue}(\lambda f. \db{e} (\lambda e. f e (\lambda x. ?)))}\\
    & \leadsto^* & \db{\lambda x. e'} {\color{blue}(\lambda f. (\db{e} (\lambda e. f e (\lambda x. ?)))}\\
    & = & (\lambda k. k (\lambda x. \db{e'})) {\color{blue}(\lambda f. (\db{e} (\lambda e. f e (\lambda x. ?)))}\\
    & \leadsto & (\lambda f. \db{e} {\color{blue}(\lambda e. f e (\lambda x. ?))}) (\lambda x. \db{e'})\\
    & \leadsto & \db{e} {\color{blue} (\lambda e. (\lambda x. \db{e'}) e (\lambda x. ?))}\\
    & \leadsto^* & (\lambda k. k\, v) {\color{blue} (\lambda e. (\lambda x. \db{e'}) e {\color{blue}(\lambda x. ?)})}\\
    & \leadsto & (\lambda e. (\lambda x. \db{e'}) e {\color{blue}(\lambda x. ?)}) \, v\\
    & \leadsto & (\lambda x. \db{e'}) v {\color{blue}(\lambda x. ?)}\\
    & \leadsto & (\db{e'}[ x:= v]) {\color{blue}(\lambda x. ?)}\\
    & \leadsto^*
\end{eqnarray*}
The step from the second line to the third is immediate if $f=\lambda x.e'$, but if not then it represents the evaluation of $f$ to its normal form $\lambda x.e'$. The continuation remains unchanged at the end of this process. Similarly, later we have the evaluation of $e$ to its normal form $v$, which is then passed to $\db{e'}$. As a result, assuming that we don't evaluate under lambda's, the transformation ensures that there is a call-by-value evaluation of the expression given to it. 

Note that if we type $x$ and $e$ as having type $A$, $e'$ as having type $B$, and $f$ as having type $A\to B$, then we can type $\db{x}$ and $\db{e}$ as having type 
$(A\to R)\to R$, $\db{e'}$ as having type $(B\to R) \to R$, 
and $\db{f}$ as having type $A\to (B\to R)\to R$. So we can take the CPS-transform as mapping a type $A$ to $(A\to R)\to R$.

\section{Continuations and logic}\label{sec:ContinuationsAndLogic}

Once we look, it is easy to find continuations and continuation-passing in logic --- in particular, through readings of proofs and rules of inference that relate to processes of computation. Readers familiar with the notion of propositions-as-types \cite{Howard1980,Girard1989} will see how this view underpins the approach of this section. 

The ideas we discuss here can be traced back to the Brouwer-Heyting-Kolmogorov (BHK) interpretation of intuitionistic logic. In this interpretation proofs, or justifications, are interpreted as mathematical objects. A justification of $A\supset B$ is a process for transforming a justification of $A$ into a justification of $B$, and when we translate this into a formal mathematical model of logical deduction, this is reified as a function. Similarly a justification of $A\wedge B$ encapsulates justifications of $A$ and $B$, and is reified as an ordered pair. As a result the operations for manipulating standard connectives can be seen as instances of standard mathematical constructions. The formalization of these processes took place roughly in tandem and in the same broad community and so it is not surprising to see a degree of convergent evolution, culminating in the well-known 
propositions-as-types correspondence (see \cite{Howard1980,Girard1989} for convenient and appropriate summaries), in which we have correspondences between propositions and types, and between justifications of propositions and elements of types. Later, a third party entered the room, and we see a further link to the type systems of functional programming languages, and expressions in the language.  

Let us adopt this viewpoint and examine the $\vee$-elimination rule:  
\[\begin{prooftree}
    \hypo{\begin{array}{c} \\ 
    \phi \vee \psi 
    \end{array}}
    \hypo{\begin{array}{c}[\phi]\\ \chi\end{array}}
    \hypo{\begin{array}{c}[\psi]\\ \chi\end{array}}
    \infer3{\chi}
\end{prooftree}\]
This rule can be justified by the argument that a proof such as 
\[
\begin{prooftree}
    \hypo{}
    \ellipsis{}{\phi}
    \infer1[$\vee$I]{\phi \vee \psi}
    \hypo{[\phi]}
    \ellipsis{}{\chi}
    \hypo{[\psi]}
    \ellipsis{}{\chi}
    \infer3[$\vee$E]{\chi}
\end{prooftree}
\qquad\qquad\mbox{can be reduced to}\qquad\qquad 
\begin{prooftree}
    \hypo{}
    \ellipsis{}{\phi}
    \ellipsis{}{\chi}
\end{prooftree}
\]
We can see here that the conclusion of the overall proof is $\chi$, and thus $\chi$ is the analogue of the result of the computation. That means we can interpret a proof element $\begin{array}{c}[\phi] \\ \chi \end{array}$ as a form of continuation. 

Once we do that we can go further. We can move the formulae in the proof rule to put it in the form: 
\[
\begin{prooftree}
    \hypo{\begin{array}{c}[\phi]\\\chi\end{array}}
    \hypo{\begin{array}{c}[\psi]\\\chi\end{array}}
    \infer2{\begin{array}{c}[\phi \vee \psi]\\\chi\end{array}}
\end{prooftree}
\]


We make a further change in notation, replacing $\begin{array}{c}[\phi]\\ \chi \end{array}$ by $\contn{\phi}{\chi}$, which we use to represent a continuation taking a proposition $\phi$ and producing an end result $\chi$. We can then reformulate the standard intuitionistic proof rules as introductions of propositions and introductions of continuations. In doing this we have changed the direct value-based formulations of $\wedge$ and $\to$ elimination into formulations that manipulate continuations, inverting the rules in the process.  

\[
\setlength{\arraycolsep}{1.5cm}
\begin{array}{cc}
\begin{prooftree}
    \hypo{\phi}\hypo{\psi}\infer2{\phi\wedge \psi} 
\end{prooftree} &
\begin{prooftree}
    \hypo{\contn{\phi}{\chi}}\infer1{\contn{\phi\wedge \psi}{\chi}} 
\end{prooftree}\hspace{3em}
\begin{prooftree}
    \hypo{\contn{\psi}{\chi}}\infer1{\contn{\phi\wedge \psi}{\chi}} 
\end{prooftree}\\
\\
\begin{prooftree}
    \hypo{\phi}\infer1{\phi\vee \psi}
\end{prooftree}\hspace{3em}
\begin{prooftree}
    \hypo{\psi}\infer1{\phi \vee \psi}
\end{prooftree} & 
\begin{prooftree}
    \hypo{\contn{\phi}{\chi}}\hypo{\contn{\psi}{\chi}}\infer2{\contn{\phi\vee \psi}{\chi}}
\end{prooftree}\\
\\
\begin{prooftree}
    \hypo{}\infer1{\top}
\end{prooftree}\\
\\
&\begin{prooftree}
    \hypo{}\infer1{\contn{\bot}{\chi}}
\end{prooftree}\\
\\
\begin{prooftree}
    \hypo{\contn{\phi}{\psi}}
    \infer1{\phi\to \psi}
\end{prooftree}& 
\begin{prooftree}
    \hypo{\phi}
    \hypo{\contn{\psi}{\chi}}
    \infer2{\contn{\phi \to \psi}{\chi}}
\end{prooftree}
\end{array}
\]
This formulation exhibits a strict symmetry between conjunction and disjunction, true and false. 

We can also see continuations in the reductive view of proof, a viewpoint that links strongly to both proof-search and logic programming. Indeed, reductive proof can be seen as a continuation-passing style account. 
In deductive proof, we combine complete proof-objects, and the active part of the proof is at the bottom of the proof tree. By contrast in reductive proof the object representing the current state of the computation is an incomplete proof culminating in the ultimate goal, and so is recognizable as a continuation. In this case the active part of the structure is at the top of the proof-tree. 

\newcommand{\query}{\;{?\!-}\;} 

As an example, employing a sequential presentation, consider the logic program (loosely following the set-up in \cite{miller1989logical}, which lies within the 
scope of base-extension semantics \cite{GP-DF-NAF-B-eS-2023})
\[
\Gamma = a_1 \supset a \,,\, a_2 \supset a \,,\, a_2 
\]
with goal 
\[
    \Gamma \query a 
\]

Resolution of $a$ against $\Gamma$ produces two matches, with 
$a_1\supset a$ and $a_2\supset a$, giving us the two incomplete proofs $C_1$ and $C_2$ below, both of which can be regarded as continuations generated from the original query (the orginal continuation for the process): 
\[C_1 = \begin{prooftree}
    \hypo{\mbox{\rm succeed}}
    \infer1{\Gamma\query a_1\supset a}
    \hypo{\Gamma\query a_1}
    \infer2{\Gamma\query a}
\end{prooftree}\mbox{\ \ \  and\ \ \  }
C_2 = \begin{prooftree}
    \hypo{\mbox{\rm succeed}}
    \infer1{\Gamma\query a_2\supset a}
    \hypo{\Gamma\query a_2}
    \infer2{\Gamma\query a}
\end{prooftree},\]
We can view these as living in a pseudo-Python term
\[\mbox{\tt try: $C_1$ except fail: $C_2$}\]
Trying to develop $C_1$, we resolve $a_1$ against $\Gamma$, which fails: 
\[
\begin{prooftree}
    \hypo{\mbox{\rm succeed}}
    \infer1{\Gamma\query a_1\supset a}
    \hypo{\rm{fail}}
    \infer1{\Gamma\query a_1}
    \infer2{\Gamma\query a}
\end{prooftree}
\]
So, we can regard this as raising the exception {\tt fail}, which is then caught by the exception handler. In the continuations-based account, this error results in abandoning the current continuation, $C_1$, and swapping in $C_2$ as the new continuation. This is then developed, resolving $a_2$ against $\Gamma$ which now succeeds and we are left with a complete proof: 
\[
\begin{prooftree}
    \hypo{\mbox{\rm succeed}}
    \infer1{\Gamma\query a_2\supset a}
    \hypo{\rm{succeed}}
    \infer1{\Gamma\query a_2}
    \infer2{\Gamma\query a}
\end{prooftree}
\]

As a result, the flow of control in a proof-search can be seen as a form of continuation passing, using the continuations to implement the raising and handling of exceptions allowing non-local control flow. 






So, fragments of the standard presentation of intuitionistic propositional logic can be presented naturally as versions of continuations, and pushing that analogy leads to structural insights that are not at first apparent. 

Now, the standard logical rules allow arbitrary formulae in the role of the result of a continuation. We have seen that this is linked to a problem in proving completeness. So, it seems reasonable only to allow atoms in that role. In other words, we restrict complete proofs to be proofs of atoms. We will take compound propositions as being equivalent if they are indistinguishable on this basis. This amounts to equating formulae if they cannot be distinguished in the context of the proof of an atom. This means quotienting the lattice we use for representing our semantics, or, alternatively, as we shall see, passing our formulae through a form of CPS-transform.



\section{Nuclei and Heyting algebra quotients}\label{sec:Nuclei}

We now show how to move from a Heyting algebra semantics to a less discriminating one based on it. For an authoritative account of this, see Johnstone \cite{johnstone1982stone}.

Suppose $H$ and $K$ are complete Heyting algebras, and $p: H\to K$ is 
an order-preserving map from $H$ onto $K$, presenting $K$ as a quotient of $H$, then under reasonable conditions there is an order-preserving map back again $i: K\to H$ such that $p\circ i = 1_K$, and so giving an embedding of $K$ in $H$. If we let $j = i \circ p : H \to H$, then we can view $j$ as giving an approximation of an element of $H$ by an element of $K$. There are two ways of doing this: a best over-approximation in which 
$(i \circ p) h = \bigvee \{ h' \mid ph' \leq ph\}$, and a best under-approximation in which $(i\circ p) h = \bigwedge \{ h'\mid ph \leq ph'\}$. In the language of category theory, these correspond to left and right adjoints of $p$, but in the context of partial orders this structure is often called a `Galois connection'. We will be concerned with the best over-approximations, and approach them by characterizing the functions $j : H \to H$ that produce them. 

\smallskip 

\begin{definition}[nucleus]\label{def:nucleus} 
If $H$ is a complete Heyting algebra, then a {\em nucleus} on $H$ is a function $j: H\to H$ such that $j$ is \smallskip
\begin{enumerate}
    \item order-preserving: if $a\leq b$, then $j(a) \leq j(b)$
    \item increasing: $a\leq j(a)$
    \item idempotent: $j^2 (a) = j(a)$
    \item $\wedge$-preserving: $j(a\wedge b) = ja \wedge jb$. 
\end{enumerate} \smallskip 
If $j$ is a nucleus on $H$, then we write $H_j$ for the set of 
fixpoints of $j$,  
$\{a \in H \mid a = j(a)\}$, which is of course also the image of $j$, $\{ j(a) \mid a \in H\}$. \fillBox
\end{definition}

\smallskip 

\begin{lemma}
    $H_j$ is a complete Heyting algebra. 
\end{lemma}
\begin{proof}
    $H_j$ inherits its top element and meets from $H$. First $\top\in H_j$: $\top \leq j \top \leq \top$, and hence $\top = j \top$. Furthermore, if $a,b\in H_j$, then so is $a\wedge b$: since
    $j (a\wedge b) = j(a) \wedge j(b) = a \wedge b$. 
    Moreover, if $\{a_\alpha \mid \alpha\in A\}$ is an arbitrary set of elements in $H$, $b\in H_j$, then $a_\alpha \leq b$ for all 
    $\alpha$ iff $j(\bigvee_\alpha a_\alpha) \leq b$. It follows that if $\{a_\alpha \}_{\alpha \in A} \subseteq H_j$, then $j(\bigvee_\alpha a_\alpha)$ is the least upper bound of $\{a_\alpha \mid \alpha\in A\}$ in $H_j$. 
    It remains to show that meet distributes over sup, but
    $j(\bigvee_\alpha a_\alpha) \wedge b = j(\bigvee_\alpha a_\alpha) \wedge j(b) = j((\bigvee_\alpha a_\alpha) \wedge b)
    = j(\bigvee_\alpha (a_\alpha \wedge b))$, as required. 
\end{proof}

\begin{lemma}
    If $H$ is a complete Heyting algebra, and if $j$ is a nucleus on $H$, then,  for any $a, b\in H$, if $b\in H_j$, then $a\to b \in H_j$.  Hence, if $a$ and $b$ are both in $H_j$, then so is $a\to b$, and it is also the exponential in $H_j$. 
\end{lemma}
\begin{proof}
   We need to show $j(a\to b) \leq a\to b$. Now $a\wedge j(a\to b) \leq j(a) \wedge j(a\to b) = j(a\wedge (a\to b)) \leq j (b) = b$, and the result follows.  
\end{proof}

The following two lemmas are both proved by easy checks on the conditions for a nucleus: 

\smallskip

\begin{lemma}
    If $h$ is any element of a complete Heyting algebra $H$, then $\lambda a. (a\to h)\to h$ is a nucleus on $H$. \fillBox
\end{lemma}

We will later use this in the case that $h$ is the interpretation of an atom, $b$, so that this nucleus is $\lambda a. (a\to \db{b})\to \db{b}$.


\smallskip 

\begin{lemma}
    If $H$ is a complete Heyting algebra, and $(j_\alpha \mid \alpha\in A)$ is a family of nuclei on $H$, then $\bigwedge_\alpha j_\alpha$ is also a nucleus on $H$. \fillBox
\end{lemma}

\smallskip 

If $\At$ is a collection of atoms, $H$ is a (complete) Heyting algebra, and $\sigma : \At\to H$ is a function giving the interpretation of atoms, then there is a canonical interpretation of intuitionistic propositional logic with atoms from $\At$ in $H$. We use $(-)_H$ to distinguish between propositional operators and the lattice operations in $H$: 
\[
\begin{array}{lr@{\quad}c@{\quad}l}
(\mbox{At}) & \db{a} & = & \sigma a \\ [1mm]
(\top)      & \db{\top} & = & \top_H \\ [1mm]
(\wedge)    & \db{\phi\wedge\psi} & = &
              \db{\phi}\wedge_H\db{\psi} \\ [1mm]
(\supset)   & \db{\phi\supset\psi} & = & 
              \db{\phi}\to_H\db{\psi} \\ [1mm]
(\bot)      & \db{\bot} & = & \bot_H \\ [1mm]
(\vee)      & \db{\phi\vee\psi} & = & 
              \db{\phi}\vee_H\db{\psi}
\end{array}
\]
In the case that $H=\worlds^\uparrow$, this corresponds to the standard Kripke interpretation. 

If $j$ is a nucleus on $H$, then, since $H_j$ is also complete Heyting algebra, there is a canonical interpretation on $H_j$ arising from $j\circ \sigma$. Since $H_j$ can be regarded as a subset of $H$, we can regard this as a non-standard interpretation in $H$:
\[
\begin{array}{l@{\quad}@{\quad}c}
(\mbox{At}) & \db{a}_j = (j \circ \sigma) a \\ [1mm]
(\top)      & \db{\top}_j = j(\top_H) = \top_H \\ [1mm]
(\wedge)    & \db{\phi\wedge\psi}_j =   
              j(\db{\phi}_j\wedge_H\db{\psi}_j) = \db{\phi}_j\wedge_H\db{\psi}_j \\ [1mm]
(\supset)   & \db{\phi\supset\psi}_j = j 
              (\db{\phi}_j\to_H\db{\psi}_j) = \db{\phi}_j\to_H\db{\psi}_j \\ [1mm]
(\bot)      & \db{\bot}_j = j (\bot_H) \\ [1mm]
(\vee)      & \db{\phi\vee\psi}_j = 
              j(\db{\phi}_j\vee_H\db{\psi}_j)
\end{array}
\]
Sandqvist's semantics arises from an instance of this form of interpretation. 

\section{Sandqvist's implicit CPS transform and completeness}\label{sec:Completeness2}

 We have shown, in Section~\ref{sec:Nuclei}, that for any atom $b$, the function $j_b (a) = (a\to \db{b})\to \db{b}$ is a nucleus on $\worlds^\uparrow$, 
and hence so is $J(a) = \bigwedge_{b\in\At} j_b(a) = \bigwedge_{b\in\At} (a\to \db{b})\to \db{b}$. We consider the interpretation of intuitionistic propositional logic in $\worlds^\uparrow_J$. 

First, the interpretation of atoms is unchanged. For any atom $b$, we have 
$\db{b} \leq J(\db{b}) = \bigwedge_{a\in\At} (\db{b}\to \db{a})\to \db{a} \leq (\db{b}\to \db{b})\to \db{b} = \db{b}$. Hence $J\db{b}=\db{b}$. Second, since $\worlds_J$ is closed in $\worlds$ under conjunction, the interpretation of conjunction is as in $\worlds$, and similarly for implication and true. 

\smallskip 

\begin{lemma}
Sandqvist's base-extension semantics for intuitionistic propositional logic arises as 
the interpretation in $\worlds^\uparrow$ arising from the nucleus $J(b) = \bigwedge_{a\in\At} (b \to \db{a}) \to \db{a}$.
\end{lemma}
\begin{proof}
    We run through the clauses for this interpretation showing their correspondence to the Sandqvist definition. Arguments make use of the soundness of intuitionistic reasoning in the standard interpretation in $\worlds^\uparrow$.

\smallskip 
    
\begin{itemize}[align=left]
    \item[($\At$)] $\db{a}_J = J \db{a}$: we have $\db{a} \leq J \db{a} = \bigwedge_{b\in\At} (\db{a}\to \db{b})\to \db{b} \leq (\db{a}\to \db{a}) \to \db{a} = \top\to \db{a} = \db{a}$. Hence $J \db{a} = \db{a}$, and the interpretation is standard, as in Sandqvist. \vspace{1mm} 
    \item[($\top$, $\wedge$, $\supset$)] All of these are interpreted as in a standard Kripke model in both the $J$-interpretation and in Sandqvist. \vspace{1mm}
    \item[($\bot$)] We have 
    \[
    \begin{array}{rcl} 
    \db{\bot}_J & = & J (\bot_H) \\ 
                & = & \bigwedge_{b\in\At} (\bot_H\to \db{b})\to \db{b} \\
                & = & \bigwedge_{b\in\At} (\top_H \to \db{b}) \\ 
                & = & \bigwedge_{b\in\At} \db{b}
    \end{array}
    \]
    So, $\base{B}\in\db{\bot}_J$ iff for every $b\in\At$, $\base{B}\in\db{b}$ iff for every $b\in\At$ $\derives{B}{b}$, as for Sandqvist. \vspace{1mm}
    \item[($\vee$)] We have 
    \[
    \begin{array}{rcl} 
    \db{\phi\vee\psi}_J & = & J(\db{\phi}_J\vee_H\db{\psi}_J) \\ 
    & = &\bigwedge_{b\in\At} ((\db{\phi}_J \vee \db{\psi}_J)\to \db{b})\to \db{b} \\ 
    & = & \bigwedge_{b\in\At} ((\db{\phi}_J\to \db{b})\wedge(\db{\psi}_J\to \db{b}))\to \db{b} 
    \end{array}
    \]
    So, using induction for $\phi$ and $\psi$, and $(\At)$ for $b$, $\base{B}\in\db{\phi\vee\psi}_J$ iff for every $b\in\At$, and for every $\base{C}\supseteq\base{B}$, if 
    $\support{C}{({\phi}\supset  {b})\wedge({\psi} \supset  {b})}$, then $\support{C}{{b}}$. 

    \medskip
    
    \noindent But then $\support{C}{({\phi}\supset  {b})\wedge({\psi} \supset {b})}$ iff 
    $\support{C}{{\phi} \supset {b}}$ and $\support{C}{{\psi} \supset {b}}$
    iff ${\phi}\support{C}{{b}}$ and ${\psi}\support{C}{{b}}$, which reduces this to Sandqvist's definition. 
\end{itemize} \vspace{-4mm}
\end{proof}

As a result, Sandqvist's interpretations of the propositional connectives, including $\bot$ and $\top$, can all be seen as coming from the same construction, and hence do not require different philosophical explanations. 

Moreover, we can now prove the equivalence of $\db{\Oval{$\phi$}}$ and 
$\db{\phi}_J$. If we compare this with Sandqvist's proof \cite{Sandqvist2025}, this amounts to demonstrating his condition $(\dagger)$:
\[ 
(\dagger) \mbox{  for every }\phi\mbox{ and every }
\base{B}\supseteq\base{N}\mbox{, }\support{\base{B}}
\mbox{\Oval{$\phi$}} \mbox{ iff }
\support{\base{B}}\phi 
\]
That condition is what fails if we use the original Heyting interpretation. The $J$-interpretation is required to make sure that it goes through. 

\smallskip 

\begin{proposition}\label{prop:prop-equiv-atom}
In $\worlds^\uparrow$, we have for all bases 
$\base{B}\supseteq\base{N}$, and for all propositions $\phi$, 
\[\base{B}\in\db{\Oval{$\phi$}}
= \db{\Oval{$\phi$}}_J \mbox{ iff }\base{B}\in \db{\phi}_J\]
\end{proposition}

\begin{proof}
The proof is by induction on the structure of $\phi$. 

\smallskip 

\begin{itemize}[align=left]
    \item[(At)] For any atom $a$, $\mbox{\Oval{$a$}} = a$
    \item[($\wedge,\top,\supset$)] See the discussion in Section \ref{sec:Completeness1}.
    \item[($\bot$)] For any $b\in\At$, we have 
    \begin{prooftree}
        \hypo{\bot}
        \infer1{b}
    \end{prooftree}
    and hence the rule 
    $(\Rightarrow\mbox{\Oval{$\bot$}})\Rightarrow b)$ is in 
    $\base{N}$. It follows that if $\support{B}{\Oval{$\bot$}}$ then 
    $\support{B}{b}$ for any atom $b$. Conversely if $\support{B}{b}$ for any atom $b$, then in particular $\support{B}{\Oval{$\bot$}}$.
    \item[($\vee$)] Recall that we previously showed 
    $\phi\vee\psi \to \mbox{\Oval{$\phi\vee\psi$}}$ 
    in the original Kripke interpretation. The equivalent proof holds in the $J$-interpretation. 
    To show the reverse, it suffices to establish 
$\mbox{\Oval{$\phi\vee\psi$}} \to ((\phi\to b)\wedge (\psi \to b)) \to b$ for any atom $b$ in the $J$-interpretation:
\[
\hspace{-7mm}
\begin{prooftree}
    \hypo{\Oval{$\phi\vee\psi$}}
    \hypo{[\Oval{$\phi$}]}
    \infer1[Ind]{\phi}
    \hypo{[(\phi\to b)\wedge (\psi \to b)]}
    \infer1[$\wedge E$]{\phi\to b}
    \infer2[$\to E$]{b}
    \hypo{[\Oval{$\psi$}]}
    \infer1[Ind]{\psi}
    \hypo{[(\phi \to b)\wedge (\psi \to b)]}
    \infer1[$\wedge E$]{\psi\to b}
    \infer2[$\to E$]{b} 
    \infer3[$\base{N}$]{b}
    \infer1[$\to I$]{((\phi\to b)\wedge (\psi \to b))\to b}
\end{prooftree}
\]
This now uses the elimination rule for \Oval{$\phi\vee\psi$} with an atom $b$ as target, and hence is an instance of a rule in $\base{N}$. 
\end{itemize} \vspace{-4mm}
\end{proof}

We now have the ingredients to prove completeness. Given soundness of the interpretation, it is enough to prove the following. 

\smallskip 

\begin{proposition}\label{prop:completeness}
    If $\Phi\support{N}{\psi}$, then $\Phi\vdash\psi$ is provable in intuitionistic propositional logic. 
\end{proposition}
\begin{proof}
    We proceed through three steps. We suppose 
    $\Phi\support{N}{\psi}$. 

    \smallskip 
    
    \begin{itemize}[align=left]
        \item[] Step 1: Using Proposition \ref{prop:prop-equiv-atom}, we translate this into a statement about atoms: 
        \[ 
            \{\mbox{\Oval{$\phi$}} \mid \phi\in\Phi \} \support{N}{\mbox{\Oval{$\psi$}}}
        \]
        \item[] Step 2: Using Lemma~\ref{lem:support-eq-derives-atoms}, we reduce this to atomic derivability: 
        \[
            \{\mbox{\Oval{$\phi$}} \mid \phi\in\Phi \} \derives{N}{\mbox{\Oval{$\psi$}}} 
        \] 
        \item[] Step 3: Finally, using Proposition \ref{prop:validity-reflection}, we conclude that 
        $\Phi\vdash\psi$ is provable in intuitionistic propositional logic. 
    \end{itemize} \vspace{-4mm}   
\end{proof}

Of the three steps in the proof above, only the first makes any use of the semantics of propositions, and so that is the critical one in understanding the difference between Sandqvist's semantics and the standard Kripke interpretation when it comes to completeness. 

\section{Proof-relevance} \label{sec:proof-relevance}

The arguments we have presented have all been couched in terms of derivability. But proof-theoretic semantics is concerned with the validity of arguments, and hence with a level of detail we have not covered, or at least have not addressed explicitly. 

The propositions-as-types (Curry-Howard)
correspondence gives a matching between propositions and types, proofs and terms, and reduction of proofs and evaluation of terms for a specific logic and a specific type theory. By choosing the logic and the type theory appropriately,  it can be made to seem almost vacuous at a technical level. The terms of the type theory become non-graphical representations of proof-trees. But there is more to it philosophically. The type theory gives ways to construct proofs, and these are then reflected into connectives and the proof rules for them. As a result, constructions of proofs are the direct reflections of basic operations at type-theoretic level. In the Kripke context, we end up with the standard Kripke interpretation as derived from the type theory of a category of presheaves. But Sandqvist's semantics is not that. 

We have seen that there are two ways to interpret a complete Heyting algebra $H_J$. We can look at it as an entity in its own right. This would imply that Sandqvist's model is a standard interpretation in a non-standard model. Specifically it is an interpretation in a model in which a form of completeness of the set of atoms is baked into the structure. Categorically, it would likely be the topos of sheaves for the Joyal-Tierney topology corresponding to $J$ \cite{MM1992}. We could interpret derivations and then proofs in that model. 

Alternatively, we can regard $H_J$ as a well-equipped subset of the normal world. In this context, we take propositions through an operation that makes them more complex. This is most evidently seen with atoms, where we map an atom $a$ to $\bigwedge_{b\in\At} (a\supset b)\supset b$, which, type-theoretically, would be
$\prod_{b\in\At} (a\to b)\to b$. Now there are obvious maps $a \longrightarrow (\prod_{b\in\At} (a\to b)\to b)$ and $(\prod_{b\in\At} (a\to b)\to b) \longrightarrow a$, but they are not inverses and we expect that $a$ is a retract of $\prod_{b\in\At} (a\to b)\to b$, but is not isomorphic to it. This means that we have to exercise care when we discuss the semantics of proofs. 

More delicately, we have seen that the CPS  transformation implicit in this non-standard semantics might impose a restriction on the admissible proof-reductions. We have not discussed this either. 

\subsection*{Acknowledgments} 

Pym is grateful to University College London for its
support of his sabbatical leave at the University of London’s Institute of Philosophy, and to the
Institute for its generous hosting. 

\bibliographystyle{siam}
\bibliography{references.bib}

\end{document}